\documentclass[11pt]{article}
\usepackage[utf8]{inputenc}
\usepackage[
papersize={5.5in, 8.5in},
left=0.35in,
right=0.35in,
top=0.35in,
bottom=0.5in]{geometry}
\usepackage{graphicx}
\usepackage{amsthm}
\usepackage{rotating}
\usepackage{amsfonts,amssymb,amsmath}
\usepackage[usenames]{color}

\usepackage{array}
\usepackage{float}
\usepackage{booktabs}
\usepackage[dvipsnames,usenames,table,xcdraw]{xcolor}

\graphicspath{ {images/} }
\newtheorem*{theorem}{Theorem}
\title{\bf 
A 6-chromatic odd-distance graph \\ in the plane
}
\author{\bf 
\textcolor[rgb]{.6,.3,.0}{Jaan Parts} \\
} 
\date{\normalsize \textcolor[rgb]{.6,.3,0}{Kazan, Russia, jaan\_parts@.mail.ru}}

\begin{document}

\maketitle

\pagestyle{empty}
\thispagestyle{empty}

\begin{abstract}
Two vertices of  an odd-distance graph are connected by an edge if and only if their Euclidean distance is an odd integer. We construct a 6-chromatic odd-distance graph in the plane.
\end{abstract}

\section{Introduction}

\paragraph{Motivation.}
This work was motivated by Marijn Heule (and partly by Dire Straits with their song "Money for Nothing\,"). He proposed \cite{heu2} to construct a 6-chromatic\footnote{
Hereinafter, the numeric prefix means the number of colors $k$.}
\textit{odd distance graph} (ODG), that is, a graph in which all pairs of vertices at an integer odd distance from each other are connected by an edge.  According to Marijn, this might help find a 6-chromatic \textit{unit distance graph} (UDG) in the plane.

Such a proposal is based, on the one hand, on his observation that adding edges of odd length speeds up the computation of the \textit{chromatic number} $\chi$ of some large (hard) UDGs. On the other hand, it is known that, for a \textit{proper} 4-\textit{coloring} of some UDGs embedded in $\mathbb{Q}(\sqrt{-3}, \sqrt{-11})$, so-called \textit{virtual edges} naturally appear between the vertices at an odd distance from each other and behave similarly to real edges: they forbid these vertices to have the same \textit{color}.

But there is also a reason to doubt that it will help. So far, virtual edges have only been observed in 4-coloring (but not in 5-coloring). Also, the 6-chromatic \textit{two-distance} graphs have been known for quite some time (see \cite{exoo}, \cite{par}), but so far this doesn't provide any clues for UDGs.

In our opinion, the main difficulty in finding UDGs with $\chi\ge 6$ is associated primarily with the high computational complexity of determining of $\chi$. The traditional tool is the SAT-based approach \cite{heu1}, in which information about the structure of a graph and the possibility of its $k$-coloring is translated into a \textit{propositional formula} (CNF format), which is checked by the SAT-\textit{solver}. If the graph has chromatic number $\chi$, then the formula is \textit{satisfiable} for $k=\chi$ colors ($k$-SAT solution) and is \textit{unsatisfiable} for $k=(\chi-1)$ colors ($k$-UNSAT solution). However, with an increase in the graph order and the number of colors $k$, the problem becomes more and more difficult, and SAT-solvers no longer cope with it in a reasonable time (we will denote such an indetermined result as $k$-INDET).

\paragraph{Notation.}
We use the following notation. Graphs are denoted by a capital italic letters (subscripts indicate the number of vertices), preceding numbers and lowercase italic letters indicate complex rotation and scaling multipliers. The sign "$+$" denotes a union of graphs, the symbol "$\oplus$" denotes the \textit{Minkowski sum}, and $\oplus^n G$ is the Minkowski sum of $n$ identical graphs $G$. The notation $[G, a]$ indicates the \textit{trimming} of the graph $G$ by discarding all vertices outside the disk of radius $a$. 

All considered graphs are embedded in the plane, are completely determined by their vertex coordinates, and have edges of odd length. 

For example, $i\,G$ and $-G$ mean the rotation of the graph $G$ around the \textit{origin} $(0, 0)$ by the angle $\pi/2$ and $\pi$;\; $8\,G$ means scaling up all vertex coordinates of $G$ by $8$ times;\; $\oplus^3 G= G \oplus G \oplus G$. Some base graphs have dedicated symbols:\; $T=\{(0, 0), (1, 0), (1/2, \sqrt3/2)\}$ is a 3-vertex triangular graph, $H=T \oplus (-T)=\{(0, 0), (\pm 1, 0), (\pm1/2, \pm\sqrt3/2)\}$ is a 7-vertex wheel graph. An \textit{element} of a graph is a subgraph that can be entirely embedded in a \textit{unit hexagonal} (or triangular) \textit{lattice} $L=\oplus^{\infty} H=\{n\, (1, 0)+m\, (1/2, \sqrt3/2); \; n, m\in\mathbb{Z}\}$.

\paragraph{Background.}
A 4-chromatic ODG requires at least 7 vertices \cite{heu2}, just like the UDG.

Ardal et al. showed \cite{ard} that the unit hexagonal lattice $L$ as an ODG has chromatic number $\chi=4$, moreover, in $L$, they constructed an 11-vertex subgraph $G_{11}$, using edges with four lengths $\{1, 3, 5, 7\}$ and containing a \textit{monochromatic pair} of vertices with distance 8 in a 4-coloring. This, by forming a \textit{spindle} (two mono-pairs with a common vertex and with two remaining vertices connected by an edge, which we call a \textit{rotation edge}), leads to a 21-vertex graph $G_{21}=G_{11}+\rho \, G_{11}$ with $\chi=5$. Hereinafter the symbol $\rho$ denotes any suitable rotation (any admissible rotation edge of odd length $r$ will do).

\paragraph{Starting point.}
It is natural to look for a 6-chromatic ODG among the graphs containing many copies of the graph $G_{21}$. 

We quite quickly established that the graph $G_{48289}=\oplus^2 (\oplus^8 H + \oplus^8 \rho H)= \oplus^2 (G_{217} + \rho\, G_{217})$ has $\chi\ge 6$. Further work consisted in the \textit{reduction} (of the order) of this graph, clarification of its key elements and generalization of the construction.

\section{Analysis}

\paragraph{Initial checks.}

First of all, note that we don't know the exact value of $\chi$ of 48289-vertex ODG, since SAT-checks with $k\ge 6$ colors were unsuccessful ($k$-INDET). Namely
, we tried $k\in\{6, 7, 8\}$. We'll come back to this question at the end, but for now let's see what else we found out.

We tried to reduce the number of different odd distances, keeping $\chi\ge 6$. It turned out that it is sufficient to use edges with lengths $\{1, 3, 5, 7\}$, as in the case of  Ardal’s graph \cite{ard}. A reduction in the set of permissible distances led to an increase in the computation time by about 10 times (from 1.5 to 15 hours), which is quite expected.

We also tried trimming of $G_{48289}$ (without limiting the set of possible distances). On this path, we managed to reduce the radius to 7.5 ($G_{26299}$). The subgraph $G_{24103}$ on the disk with radius 7 gave 5-INDET.

\paragraph{Tools.}

As SAT-solvers, we used together a pair of programs {\tt Glucose} and {\tt Yalsat}. The latter usually (but not always) copes better with the SAT solution, the latter with UNSAT. To reduce the graphs, we used the program {\tt DRAT-trim}. As input, {\tt DRAT-trim} uses a proof file (so-called \textit{certificate}) emitted by {\tt Glucose} during the solution and extracts \textit{unsatisfiable core}, that is, a subgraph of the original graph \cite{heu1}.

The calculations were performed mainly on a computer with processor Intel Core i5-9400F, 2.9 GHz, 6 cores/6 threads, 16 GB RAM. Basically, single-threaded versions of programs were used, which made it possible to perform several checks in parallel, as well as obtain a certificate file.

In our study, {\tt Glucose} spent 1 to 4 hours to process one average graph (UNSAT). The processing terminated forcibly not earlier than after 24 hours (INDET). 
Here lies one of the drawbacks of the SAT-solvers we know: they do not help us even approximately predict how long it would take them to complete the calculations, one day or one year. Termination timeout is a blind choice based on some experience. 


Reduction with {\tt DRAT-trim} is resource intensive. Certificate processing time was on average 4 times longer than its creation in {\tt Glucose}. The certificate files were huge and grew by an average of 2 GB per hour. Thus, using several threads, it is possible to fill all the available computer memory in a day or two in especially difficult cases.

\paragraph{Reduction.}

The very first attempts to use {\tt DRAT-trim} allowed us to decrease the number of vertices to about one and a half thousand, but further progress was much slower. We tried a series of reductions, partially restoring the rotational symmetry of the graph at intermediate stages, but this did not improve, but rather worsened the result. The reason for this became clear later, when we got a better understanding of the structure of the graph: it is simply not internally symmetrical about the origin (lattice node).

With the help of {\tt DRAT-trim} we were able to reduce the number of vertices to 1123. For further reduction, we gradually discarded the vertices that are farthest from the origin and/or have the smallest vertex degree. As a result, we got a 495-vertex graph $G_{495}$. This was already enough to analyze the structure of 6-chromatic ODGs.

\paragraph{Graph elements.}

We split the 495-vertex graph into subgraphs, each of which can be embedded (separately) in the hexagonal lattice $L$. As a result, we obtain a base subgraph (\textit{core}) and 8 oblique subgraphs (\textit{rotors}), having a common vertex with the core. We call the set of all such common vertices a \textit{frame}. It is clear, that the rest of the rotor vertices are also interconnected by many copies of the frame (or its subgraphs). The rotors are also connected to the core by several (from 3 to 6) rotation edges.

The graph $G_{495}$ includes a 197-vertex core and 8 rotors: 2 central (with 66 and 41 vertices) and 6 peripheral (from 30 to 35 vertices each).

Fig.~\ref{g495} shows some elements of $G_{495}$: a core with the frame vertices highlighted on it and two rotors, one each from the center and periphery of the graph. One can see the symmetry of both the frame (to restore it, it is enough to add another central rotor) and rotors (the asymmetry of the peripheral rotor is explained by the fact that we used trimming during reduction).

A 1123-vertex graph $G_{1123}$ contains similar (but larger and not truncated) rotors and, at a unit distance from them, 6 additional rotors with the number of vertices from 18 to 29, as well as many smaller ones.

\begin{figure}[H]
\centering
\includegraphics[scale=0.3]{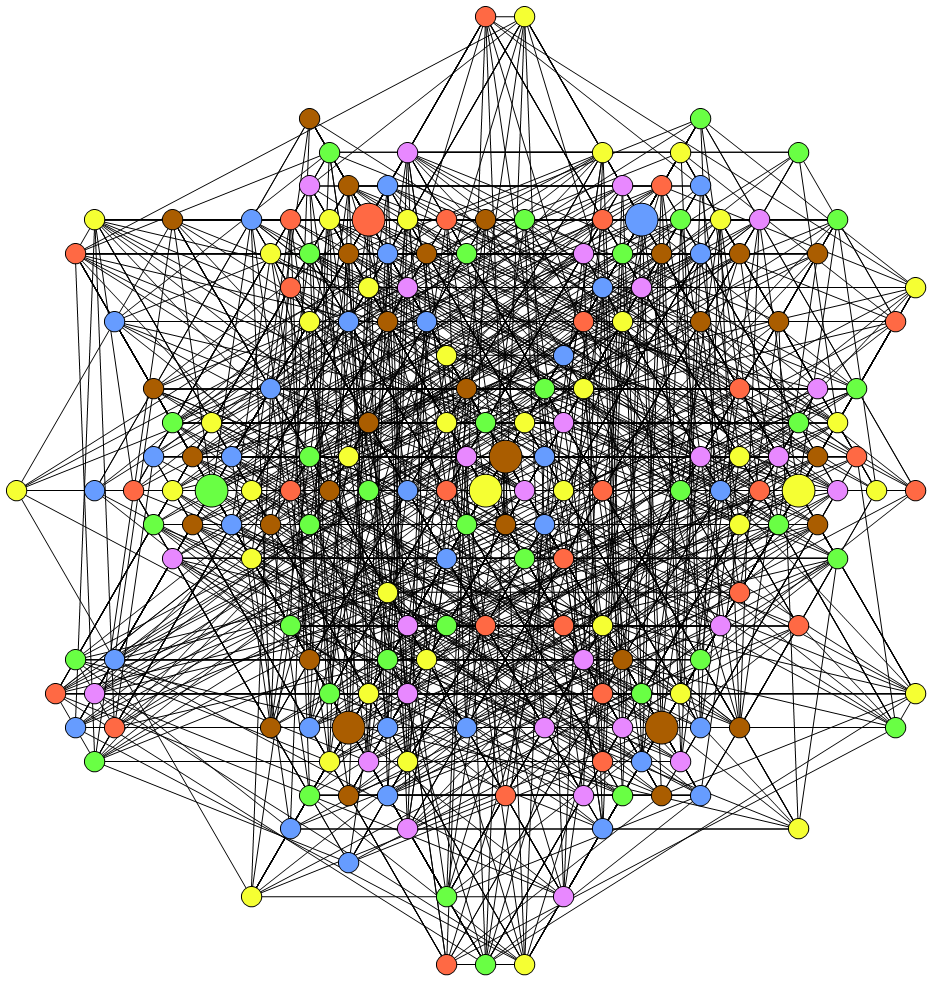} 
\includegraphics[scale=0.3]{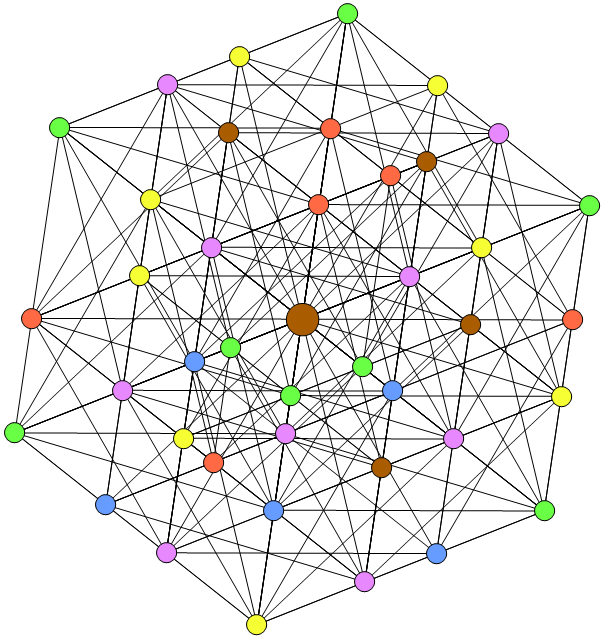} \ \ \ \
\includegraphics[scale=0.3]{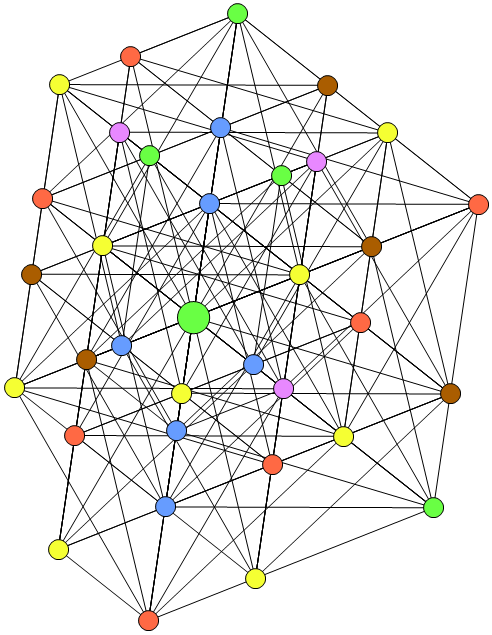} 
\caption{Some subgraphs of 6-chromatic 495-vertex odd-distance graph 
in 6-coloring: core (shown above) and rotors (below). Scale and orientation ($r=3$) are retained. The frame vertices of the core are enlarged.}
\label{g495}
\end{figure}

\section{Synthesis}

\paragraph{Main construction.}

At the stage of synthesis based on the reduced graph, we tried to guess the more optimal form of the graph elements, gradually removing vertices and adding new ones. These elements should be small, symmetrical and simple to describe.

As a result, we settled on a 306-vertex graph that includes the 9-vertex frames $F=(-7\,T)\oplus 8\,T$, the 36-vertex core $C=F\oplus 8H$, and the 31-vertex rotors $R=[\,3H\oplus 8H,\; 8\,]$. The rotation edges can be of any odd length from 1 to 15.

\paragraph{Generalization options.}

After that we tried to take other combinations of integer distances, keeping the number of vertices and the description form of the elements.

It turned out that some frames in the generalized form $F=n T\oplus m T$ with even $m$ and odd $n$ are also workable, that is, they give ODG with $\chi=6$. The same was observed for rotors 
$R=[\,n H\oplus m H,\; \max(n, m)\,]$.

We went through many pairs $(m, n)$, and empirically established the following regularity: in 6-chromatic ODGs, the frame necessarily contains, in addition to equilateral triangles with edges of odd length $n$ and $(n+m)$, two \textit{oblique} triangles with edges of odd length $s$, the vertices of which lie on the edges (or their extension) of the $(n+m)$-triangle (see Fig.~\ref{proof}~and~\ref{frame}).
In other words, triangles with edges of length $\{m, n, s\}$\footnote{
Abbreviations from mono- and non-mono-chromatic vertex pair, and skew edge.}
are realized at the points of the lattice $L$. Rotors contain similar triangles (Fig.~\ref{rotor}).

We do not know why this rule works, but to find all possible variants of working pairs $(m, n)$, we need the following

\begin{theorem}
For any integer-distance triangle with vertices at the points of the unit hexagonal lattice and edges of both even and odd length, only one of which is oriented obliquely to the base vectors of the lattice:

1. Only one edge has an even length $m$, and this edge is not oblique.

2. The length $m$ is a multiple of $8$.

3. The other two edges are of odd length, described by the expressions $s=12/t \cdot (m/8)^2+t$ and $n=(m/2) \pm (12/t \cdot (m/8)^2-t)$ for integer $t$.
\end{theorem}

\begin{proof}
The edge lengths $\{m, n, s\}$ of the triangle are related by the expression $s^2=m^2 \pm m\cdot n+n^2$, where $s$ is the length of the oblique edge. 

If both $m$ and $n$ are even, then $s^2$ is even, hence $s$ is even, which does not satisfy the conditions of the theorem. If $m$ and/or $n$ are odd, then $s^2$ and $s$ are odd, thus statement 1 is proved.

Further, let $m$ be even, $n$ and $s$ odd. Let's rewrite the expression connecting the lengths in the following form: $s^2-n^2=(s-n)(s+n)=m\, (m \pm n)$. After substitutions $s=2a+1$ and $n=2b+1$, $a, b\in\mathbb{Z}$, we get $(s-n)(s+n)=2(a-b)\cdot 2(a+b+1)\in 8\mathbb{Z}$, because for any integer $a$ and $b$, either $(a-b)$ or $(a+b+1)$ is even. Since $(m \pm n)$ is odd, this proves statement 2. 

Now take $m \in 8\mathbb{Z}_{>0}$. Draw the height $h=\sqrt3/2 \cdot m$ to the edge of length $n$, forming two right-angled triangles with edges $\{m, h, m/2\}$ and $\{s, h, l\}$, where $ l=n-m/2$ is odd, since $m/2\in4\mathbb{Z}$ (see Fig.~\ref{proof}).

From the equality $s^2-l^2=h^2$ for a right-angled triangle we obtain $(s-l)(s+l)=3/4\cdot m^2$. Since $s$ and $l$ are odd, and $s>l$, their difference is even $s-l=2\,t$, for some positive integer $t \in \mathbb{Z}_{>0}$. After substitution $l=s-2\,t$ we get $t\,(s-t)=12(m/8)^2$, where $(m/8)\in \mathbb{Z}_{>0}$, whence $s=12/t\cdot (m/8)^2+t$. Similarly, $l=12/t\cdot(m/8)^2-t $. Note that for each pair $(m, s)$ there are two different solutions for $n$ that satisfy the conditions of the theorem: $n_{\pm}=m/2 \pm l=m/2 \pm\sqrt{s^2-3/4\cdot m^2}$.

This completes the proof.

\end{proof}

\begin{figure}[H]
\centering
\includegraphics[scale=0.3]{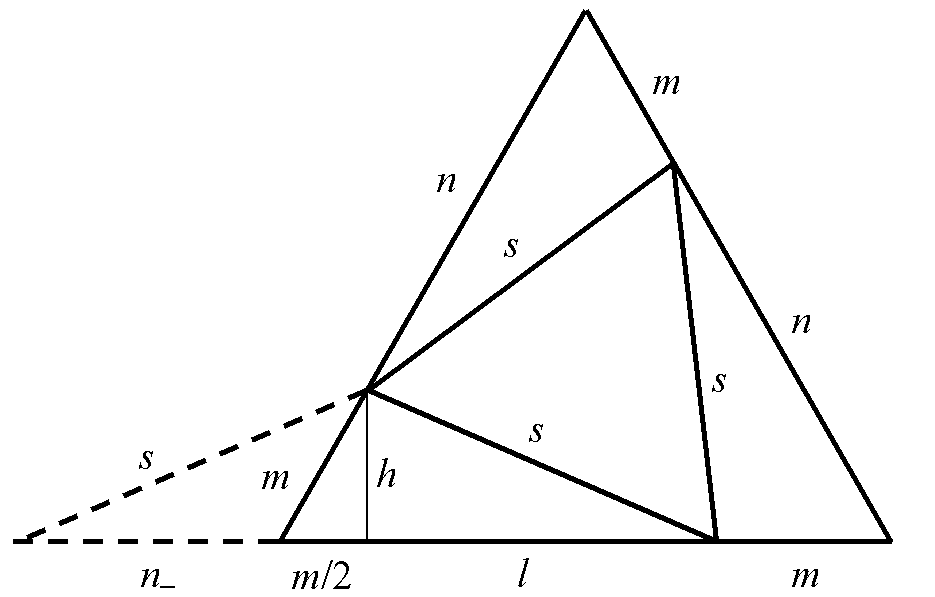} 
\caption{To the proof of the Theorem.}
\label{proof}
\end{figure}

\newpage 

To find all possible pairs $(m, s)$, it is enough to iterate over all admissible values of $m$ and $t$.
Since $m \in 8\mathbb{Z}$, the cases $t=1$ and $t=3$ are always realized and give odd $s=12(m/8)^2+1$ and $s=4(m/8)^2+3$. The cases $t\in\{2, 6, 8, 10, 14, 18, 22, 24,\dots\}$ are never realized, since they give fractional or even values of $s$. The cases $t=4$ and $t=12$ are realized only for $m \in8+16\mathbb{Z}$. Other cases of $t$ can be realized for some $m$.

The allowable values of $n$ and $s$ for some small $m$ are given in Tables \ref{tn} and \ref{ts}. The minimum value of the hypotenuse $s$ is limited by legs $h$ and $l$: $ s\ge\lceil\sqrt{3/4\cdot m^2+1}\rceil$, where $\lceil\cdot\rceil$ denotes rounding up to the nearest odd value, and is achieved for some $m$ (in the Table~\ref{ts} these values are underlined).


\begin{table}[!p]
{
\caption{All valid pairs $(m, n)$ for $m<100$.}
\label{tn}
\smallskip
\footnotesize

\begin{tabular}{@{}cc@{}}
\begin{tabular}{@{\;\;\;\;}r|l@{\:\:}}
\hline
\small{$m$} & \small{$n$}  \\
\hline
\hline

8&$-7, 3, 5, 15$\\
16&$-39, -5, 21, 55$\\
24&$-95, -21, -11, 9, 15, 35, 45, 119$\\
32&$-175, -45, 77, 207$\\
40&$-279, -77, -51, -35, 7, 15, 25,$\\& 33, 75, 91, 117, 319\\
48&$-407, -117, -15, 13, 35, 63,$\\& 165, 455\\
56&$-559, -165, -115, -49, -9, 21,$\\& 35, 65, 105, 171, 221, 615\\

\hline
\end{tabular}
&
\begin{tabular}{@{\;\;\;\;}r|l}
\hline
\small{$m$} & \small{$n$}  \\
\hline
\hline
64&$-735, -221, 285, 799$\\
72&$-935, -285, -203, -63, -33, 27,$\\& 45, 105, 135, 275, 357, 1007\\
80&$-1159, -357, -195, -25, -19, 17,$\\& 63, 99, 105, 275, 437, 1239\\
88&$-1407, -437, -315, -77, -65, 33,$\\& 55, 153, 165, 403, 525, 1495\\
96&$-1679, -525, -135, 11, 85, 231,$\\& 621, 1775\\
\\

\hline
\end{tabular}
\end{tabular}
}
\end{table}

\begin{table}[!p]
{
\caption{All valid pairs $(m, s)$ for $m\le 400$.}
\label{ts}
\smallskip
\footnotesize

\begin{tabular}{@{}cc@{}}
\begin{tabular}{r|l@{\;}}
\hline
\small{$m$} & \small{$s$}  \\
\hline
\hline

8  &\underline{7}, 13\\
16 &19, 49\\
24 &\underline{21}, 31, 39, 109\\
32 &67, 193\\
40 &\underline{35}, 37, 65, 79, 103, 301\\
48 &\underline{43}, 57, 147, 433\\
56 &\underline{49}, 61, 91, 151, 199, 589\\
64 &259, 769\\
72 &\underline{63}, 93, 117, 247, 327, 973\\
80 &73, 91, 95, 245, 403, 1201\\
88 &\underline{77}, 133, 143, 367, 487, 1453\\
96 &91, 201, 579, 1729\\
104 &\underline{91}, 169, 181, 511, 679, 2029\\
112 &\underline{97}, 133, 163, 343, 787, 2353\\
120 &\underline{105}, 111, 127, 133, 155, 195,\\& 237, 309, 545, 679, 903, 2701\\
128 &1027, 3073\\
136 &\underline{119}, 221, 301, 871, 1159, 3469\\
144 &129, 171, 259, 441, 1299, 3889\\
152 &\underline{133}, 247, 373, 1087, 1447, 4333\\
160 &\underline{139}, 217, 335, 965, 1603, 4801\\
168 &\underline{147}, 157, 183, 217, 223, 273,\\& 453, 597, 763, 1327, 1767, 5293\\
176 &169, 209, 379, 539, 1939, 5809\\
184 &\underline{161}, 299, 541, 1591, 2119, 6349\\
192 &283, 777, 2307, 6913\\
200 &\underline{175}, 185, 325, 395, 515, 637,\\& 1505, 1879, 2503, 7501\\
208 &217, 247, 523, 637, 2707, 8113\\
216 &\underline{189}, 279, 351, 741, 981, 2191,\\& 2919, 8749\\
224 &211, 241, 469, 1351, 3139, 9409\\

\hline
\end{tabular}
&
\begin{tabular}{r|l@{\:}}
\hline
\small{$m$} & \small{$s$}  \\
\hline
\hline
232 &203, 377, 853, 2527, 3367, 10093\\
240 &215, 219, 273, 285, 427, 457, 691,\\& 735, 1209, 2165, 3603, 10801\\
248 &217, 403, 973, 2887, 3847, 11533\\
256 &4099, 12289\\
264 &\underline{229}, 231, 341, 399, 429, 511, 1101,\\& 1199, 1461, 3271, 4359, 13069\\
272 &323, 337, 833, 883, 4627, 13873\\
280 &245, 247, 259, 271, 305, 349, 455,\\& 553, 613, 721, 755, 995, 1237,\\& 2107, 2945, 3679, 4903, 14701\\
288 &273, 307, 603, 1737, 5187, 15553\\
296 &259, 481, 1381, 4111, 5479, 16429\\
304 &361, 409, 931, 1099, 5779, 17329\\
312 &273, 277, 403, 507, 543, 703, 1417,\\& 1533, 2037, 4567, 6087, 18253\\
320 &331, 793, 1295, 3845, 6403, 19201\\
328 &287, 533, 1693, 5047, 6727, 20173\\
336 &\underline{291}, 301, 399, 481, 489, 811, 1029,\\& 1339, 2361, 3031, 7059, 21169\\
344 &301, 559, 1861, 5551, 7399, 22189\\
352 &313, 427, 737, 2123, 7747, 23233\\
360 &315, 333, 343, 381, 399, 465, 585,\\& 711, 927, 997, 1235, 1635, 2037,\\& 2709, 4865, 6079, 8103, 24301\\
368 &437, 577, 1127, 1603, 8467, 25393\\
376 &329, 611, 2221, 6631, 8839, 26509\\
384 &1051, 3081, 9219, 27649\\
392 &343, 427, 637, 1057, 1393, 2413,\\& 4123, 7207, 9607, 28813\\
400 &365, 455, 475, 673, 1225, 1891,\\& 2015, 6005, 10003, 30001\\
\hline
\end{tabular}
\end{tabular}
}
\end{table}


\paragraph{Infinite family.}

It follows from the Theorem that, for an arbitrary $m\in 8\mathbb{Z}_{>0}$, it is always possible to implement at least four triangles with different $\{m, n, s\}$. In particular, one can always get a triangle with edges of maximum length $s=12(m/8)^2+1$ and $n=12 (m/8)^2-1+m/2$.

Thus we have an infinite family of 306-vertex 6-chromatic graphs, including graphs with arbitrarily large odd distances.

To show this, it suffices to determine the chromatic number of one graph of the family and demonstrate that the other graphs are isomorphic to it. However, strictly speaking, we have several non-isomorphic subfamilies of 306-vertex graphs, which manifests itself in their different behavior upon further reduction. We did not count the exact number of subfamilies, but in any case there should not be many. The isomorphism of the graphs obtained for $t=1$, starting from some $m$, is quite obvious.

The general formula of the considered family of graphs: 
$$G_{306}=F(m_1, n_1) \oplus \left(m_1 H+\rho(r) R(m_2, n_2)\right),$$
where $F(m_1, n_1)=n_1 T\oplus m_1 T$, \ \ $R(m_2, n_2)=[\,n_2 H\oplus m_2 H, \max(n_2, m_2)\,]$, $n_2>0$, \ $\rho(r)=\exp{(i \arccos\frac{m_1^2+m_2^2-r^2}{2 m_1 m_2})}$, \
$|\,m_1-m_2\,|<r<|\,m_1+m_2\,|$,
and $r$ is the length of the rotation edge. 
That is, a specific graph is uniquely determined by five parameters $\{m_1, m_2, n_1, n_2, r\}$. This gives a rich set of different options for embedding a graph into extensions of the field $\mathbb{Q}$.

\paragraph{Construction details.}

The resulting graph $G_{306}$ contains the core $C$, 9 copies of the rotor $R$, 30 copies of the frame $F$, as well as rotation edges. The core $C$ is necessary: if we replace it with $F$, then the resulting 279-vertex ODG will have $\chi<6$.

The frame $F$ has 9 vertices and 21 edges with lengths $|\,n\,|$ (9 edges), $s$ (6), $|\,m-n\,|$ (3), $|\,m+n\,|$ (3). Note that $n$ can be negative here.

The rotor $R$ has 31 vertices and at least 168 edges with lengths $s$ (60), $|\,n\,|$ (42), $|\,n-m\,|$ (42), $|\,n+m\,|$ (12), $|\,n-2m\,|$ (12). Here $n>0$. Each rotor with $(m, n) = (24\,a, 35\,a)$ and odd $a$ has 12 additional oblique edges of length $s_2=\sqrt{n^2-n\cdot 2m+(2m)^2}$.

The core $C$ has 36 vertices and at least 162 edges with lengths $s$ (48), $|\,n\,|$ (36), $|\,n\pm m\,|$ (24 each), $|\,n\pm 2m\,|$ (12 each), $|\,n\pm 3m\,|$ (3 each). In tested cores, the number of edges did not exceed 192. For example, cores with $m\in 16\mathbb{Z}_{>0}$ have 162 edges. Cores with $(m, n) \in\{(8\,a, 5\,a), (24\,a, 35\,a)\}$ and odd $a$ have 24 additional edges of length $(3m-n)$ and $s_2$, respectively. Cores with $(m, n) = (24\,a, -11\,a)$ and odd $a$ have 30 additional edges of length $\sqrt{n^2+3m^2}$. We also observed additional oblique edges of length $|\,3n\,|$, $|\,n+m\,|$, $|\,n+2m\,|$, $\sqrt{n^2+4n m+7m^2}$, etc.

Of the 162 rotation edges, only one third (54) connects the rotors to the core, the rest (108) connect the rotors to each other.

Additional edges are not necessary, although they are useful for further graph reduction. One can also remove edges of length $|\,n\pm 3m\,|$. Thus, to get a 306-vertex 6-chromatic ODG, it is sufficient to use edges of five lengths: $s$ (768), $|\,n\,|$ (684), $|\,n-m\,|$ (492), $|\,n+m\,|$ (222), $|\,n-2m\,|$ (120), plus 162 rotation edges of any suitable length $r$ of the five listed (2448 edges in total).

\vspace{10mm}

\begin{figure}[H]
\centering
\includegraphics[scale=0.31]{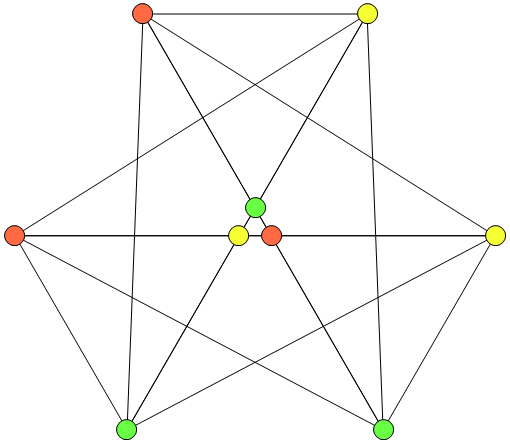} \ \ \ \
\includegraphics[scale=0.31]{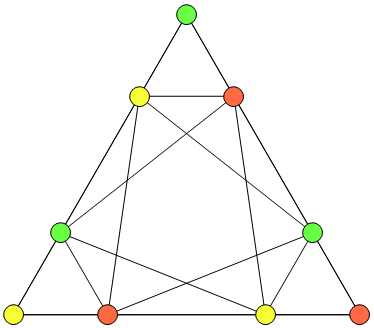} 
\caption{Examples of 9-vertex frame $F$ and their 3-coloring, $(m, n, s)=(8,-7,13)$ and $(8,3,7)$. Scales are the same.}
\label{frame}
\end{figure}

\begin{figure}[H]
\centering
\includegraphics[scale=0.31]{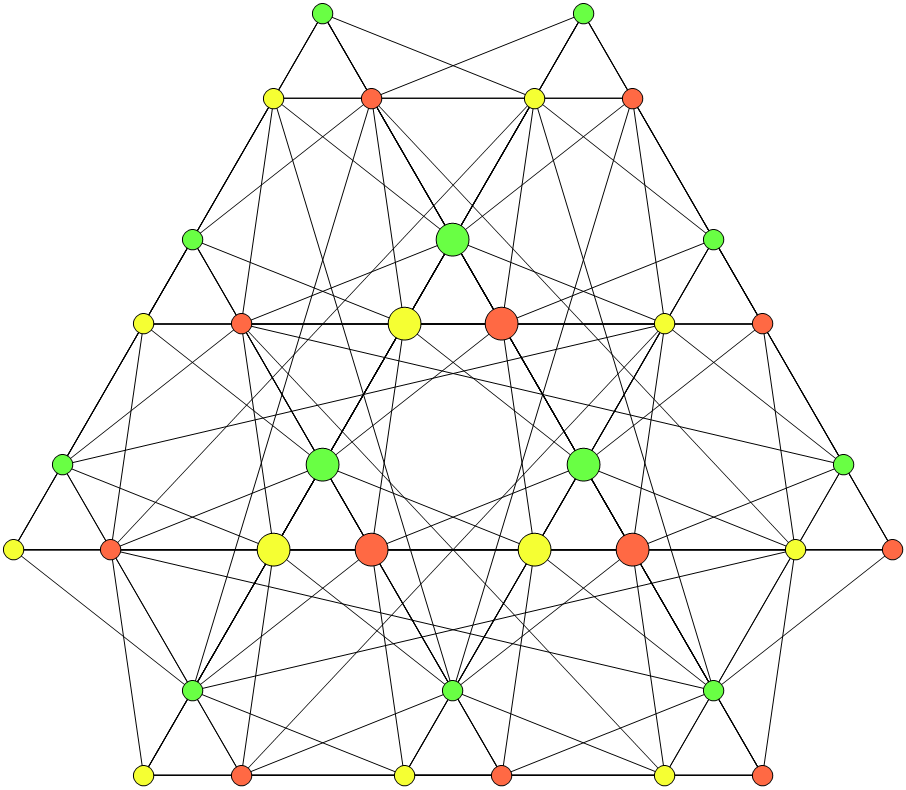} 
\caption{Example of 36-vertex core $C$ and its 3-coloring, $(m, n, s)=(8,3,7)$. Vertices common to the core and rotors are enlarged.}
\label{core}
\end{figure}

\vspace{5mm}

\begin{figure}[H]
\centering
\includegraphics[scale=0.31]{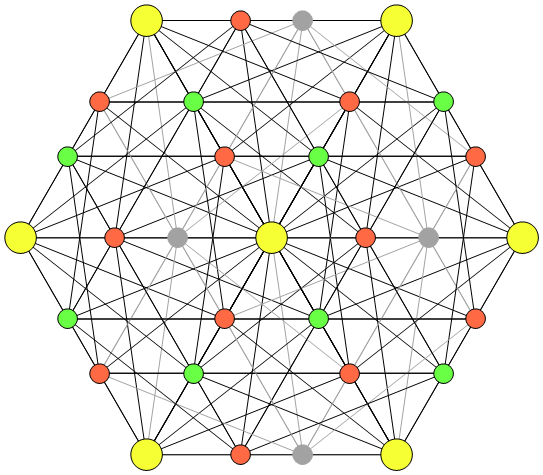} \
\includegraphics[scale=0.295]{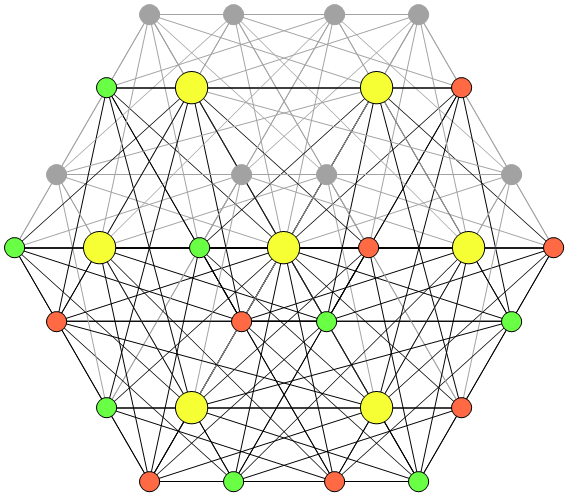} 
\caption{Examples of 31-vertex rotor $R$ and their 3-coloring, $(m,\, n,\, s\, [ , s_2 ])=(8,3,7)$ and $(24,35,31,43)$. Scales are different, rotation is not shown. Vertices connected to the core are enlarged. Vertices that can be removed are grayed.}
\label{rotor}
\end{figure}

\newpage 

\paragraph{Further reduction.}

Removing one vertex of the frame (that is, the entire rotor), as well as replacing the core with the frame, leads to a decrease in the chromatic number. Further reduction of the graph is possible by discarding some vertices of the rotors. We considered only the reduction option in which all rotors are the same, including their orientation.

In different cases, it is possible to remove from 2 to 8 vertices. For $n_2 <2\,m_2$, one can remove 2 or 4 vertices, depending on the parameters of the core. For $n_2> 2\,m_2$, one can discard 4 vertices with any core. The discarded rotor vertices gather in rhombuses with side $m_2$ and do not belong to the subgraph $m_2 H$, all the vertices of which are connected to the core (see Fig.~\ref{rotor}).
The family of rotors with $(m_2, n_2)=(24\,a, 35\,a)$ and odd $a$ allows 8 vertices to be removed with any core. In the latter case, the number of rotor vertices decreases to 23, this gives us a 234-vertex 6-chromatic ODG (see Fig.~\ref{g234}).

\vspace{2mm}

\begin{figure}[H]
\centering
\includegraphics[scale=0.31]{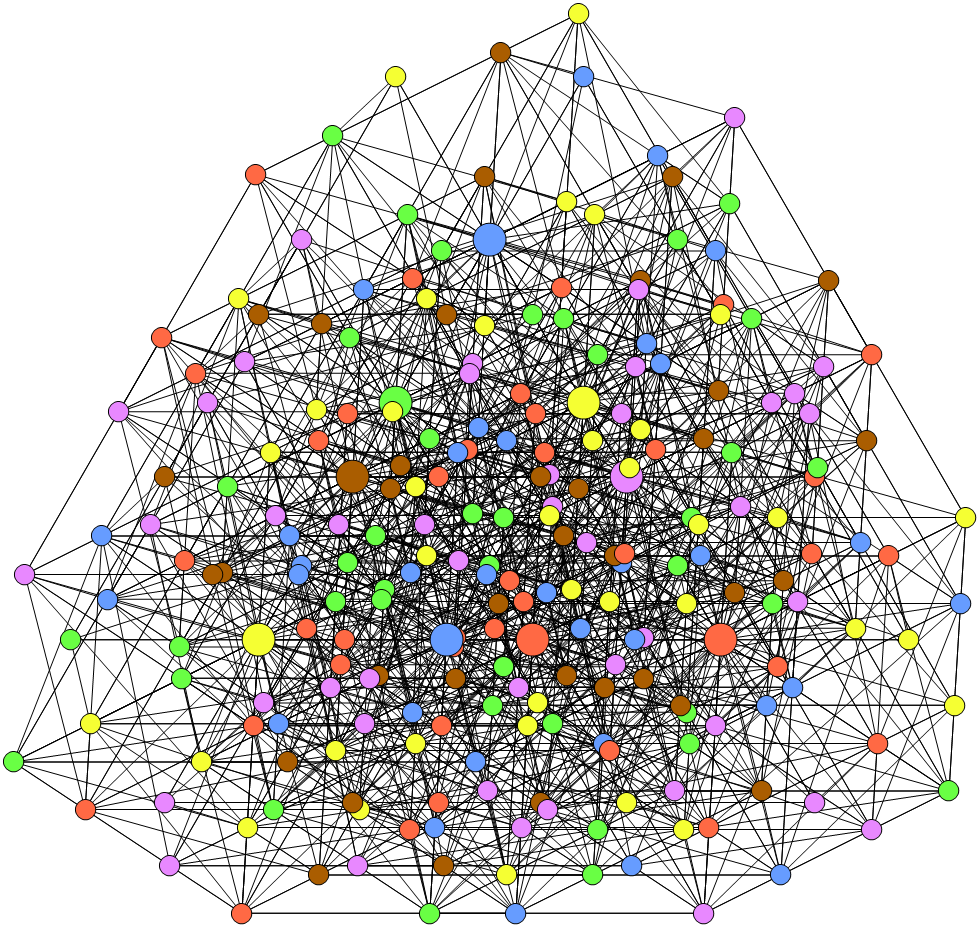} 
\caption{The 6-chromatic 234-vertex odd-distance graph $G_{234}$ and its 6-coloring, $(m,\, n,\, s,\, s_2)=(24,35,31,43)$. Edge lengths are 11, 13, 31, 35, 43, and 59. Frame vertices of the core are enlarged.}
\label{g234}
\end{figure}

\section{Postscripts}

\paragraph{On inverse distances.}

In his article \cite{heu2}, Marijn Heule states: “if an odd distance $d$ is a $k$-virtual edge of a subset $S$ of the Euclidean plane, then distance $1/d$ is a $k$-virtual edge in $S$ as well".

This statement is partially confirmed by pairs of distances of the form $3^{a}$ and $3^{-a}$, $a\in\mathbb{Z}$, which are indeed 4-virtual edges in $\mathbb{Q}(\sqrt{-3}, \sqrt{-11})$. But does this duality work in the general case?

It is not difficult to give a finite counterexample in which this duality is not observed. Take $d=5$ and a 2035-vertex graph $G_{2035}=V\oplus V\oplus H$, where $ V=\sum_{n=-2}^2 \exp(i\,\frac{n}{2} \arccos\frac{5}{6})H$. The distance $d=5$ is obviously a 4-virtual edge here, since none of the 132 pairs of vertices with this distance can be monochromatic in any 4-coloring. The distance $1/d=1/5$ is not realized in $G_{2035}$. In a combined graph containing two copies of $G_{2035}$, shifted relative to each other by $1/5$, none of the 2035 distances $1/5$ is a virtual edge.

In a private correspondence, Marijn clarified that he meant not a finite graph, but larger objects like $\mathbb{Q}(\sqrt{-3}, \sqrt{-11})$. But even in this case, doubts remain about the correctness of the statement. Therefore, a more accurate formulation of the statement and its proof are of certain interest.

\paragraph{Summary.}

We confirmed the existence of 6-chromatic ODGs in the plane, determined that it is enough to use edges of four lengths $\{1, 3, 5, 7\}$, found a simple construction, changing the parameters of which we can get an infinite family of 6-chromatic 306-vertex ODGs, and reduced the required number of vertices to 234. The graph $G_{234}$ is not vertex critical and allows further reduction: at least half of the vertices of the core can be removed painlessly. So, at the time of writing, the smallest 6-chromatic ODG has 216 vertices. Presumably, this is far from the limit.

As the graph order decreases, the number of odd distances required to preserve $\chi=6$ naturally increases: from four for the original 48289-vertex graph to 5 and 6 for 306- and 234-vertex graphs, respectively.

Similarly, adding vertices allows new valid options. For example, when using a 49-vertex rotor $R=n_2 H\oplus m_2 H$, the previously forbidden pair $(m_2, n_2)=(8, 7)$ also leads to a 6-chromatic graph of order 468.

\paragraph{More colors.}

It is quite remarkable that frame $F$, core $C$, and rotor $R$ used in $G_{306}$ have the chromatic number $\chi=3$ (see Fig.~\ref{frame},~\ref{core},~and~\ref{rotor}). Nevertheless, they lead to a graph with $\chi=6$. It is natural to predict that, using the graph elements with $\chi=4$ (which is not difficult), one can get a graph with $\chi>6$. So far we have not succeeded.

Here we ran into serious technical difficulties. So, a 2485-vertex graph $G_{2485} = G_{49} \oplus (m\,H+\rho\,G_{49})$ based on 49-vertex elements of the same form $G_{49}=n\,H \oplus m\,H$ with $\chi=4$ gave 6-INDET. The same result was observed for a similar graph $G_{2527} = G_{49A} \oplus (8\,H+\rho\,G_{49A})$, where subgraph $G_{49A}=[\,(3H+i\sqrt{27}H) \oplus 8\,H,\;8\,]$ contains many copies of Ardal's mono-pair $G_{11}$.
We also could not determine the chromatic number of the 1224-vertex graph $G_{1224} = G_{306} \oplus n\,H$. Moreover, we tried the modified 306-vertex graph with non-integer edges added for testing purposes, and got 6-INDET even for two new distances. That is, for six colors, we stumbled upon the limitation of SAT-solvers.


Not surprisingly, none of the checks of the original 48289-vertex graph with the number of colors $k = 6, 7, 8$ gave determined results.

One could still dig here. And the excavation promises to be impressive, since, according to Soifer’s conjecture \cite{soi}, there are ODGs in the plane with infinitely large $\chi$. However, even for the 7-chromatic ODG, the award has not yet been announced, so this is a pointless exercise. (And here we look at Marijn inquiringly.)


\begin{thebibliography}{100}
\bibitem{ard}
H. Ardal, J. Ma\v{n}uch, M. Rosenfeld, S. Shelah, and L. Stacho, The odd-distance plane graph.
\textit{Discrete} \& \textit{Computational Geometry}, vol.~42, no.~2, 2009, pp.~132--141.

\bibitem{exoo}
G. Exoo and D. Ismailescu, A 6-chromatic two-distance graph in the plane.
\textit{Geombinatorics}, vol.~29, no.~3, 2020, pp.~97--103.


\bibitem{heu1}
M.J.H. Heule, Computing small unit-distance graphs with chromatic number 5.
\textit{Geombinatorics}, vol.~28, no.~1, 2018, pp.~32–-50.

\bibitem{heu2}
M.J.H. Heule, Odd-distance virtual edges in unit-distance graphs.
\textit{Geombinatorics}, vol.~31, no.~2, 2021, pp.~68--76.

\bibitem{par}
J. Parts, A small 6-chromatic two-distance graph in the plane.
\textit{Geombinatorics}, vol.~29, no.~3, 2020, pp.~111--115.

\bibitem{soi}
A. Soifer, The Hadwiger-Nelson problem. In J.F. Nash, Jr. and M.Th. Rassias, editors,
\textit{Open Problems in Mathematics}, Springer, 2016, pp.~439--457.

\end{thebibliography}
\end{document}